\documentclass{amsart}
\usepackage{amsfonts}
\usepackage{amssymb}
\usepackage{amsmath}  
\usepackage{hyperref}

\usepackage{stmaryrd}
\usepackage[retainorgcmds]{IEEEtrantools}

\input xy
\xyoption{all}

\theoremstyle{plain} 
\newtheorem{theorem}{Theorem}[section]
\newtheorem{lemma}[theorem]{Lemma}

\newtheorem{proposition}[theorem]{Proposition}
\newtheorem{corollary}[theorem]{Corollary}

\newtheorem*{theostar}{Theorem}
\newtheorem*{corostar}{Corollary}

\theoremstyle{definition} 
\newtheorem{example}[theorem]{Example}
\newtheorem{examples}[theorem]{Examples}
\newtheorem{definition}[theorem]{Definition}
\newtheorem{remark}[theorem]{Remark}

\newcommand{\Ker}{\mbox{\rm Ker\,}}

\newcommand{\End}[1]{\operatorname{\rm End}_{#1}}

\newcommand{\Hom}[1]{\operatorname{{\rm Hom}}_{#1}}

\newcommand{\Ext}[2]{\operatorname{\rm Ext}^{#1}_{#2}}
\newcommand{\rad}{\mbox{ \,\rm rad\,}}

\newcommand{\dimv}{\underline{\dim}\,}

\newcommand{\add}{\mbox{{\rm add \!}}}

\newcommand{\MOD}{\mbox{{\rm mod \!}}}
\newcommand{\fgmod}{\mbox{{\rm f.g.mod \!}}}

\newcommand{\ind}[1]{\operatorname{\rm ind}_{#1}}

\newcommand{\Gr}[1]{\mbox{{\rm Gr}}_{#1}}

\newcommand{\demo}[1]{\textsc{Proof.} #1 \hfill $\Box$ \bigskip}

\newcommand{\cA}{\mathcal{A}}

\newcommand{\cC}{\mathcal{C}}

\newcommand{\cE}{\mathcal{E}}

\newcommand{\cR}{\mathcal{R}}

\newcommand{\cU}{\mathcal{U}}

\newcommand{\cY}{\mathcal{Y}}

\newcommand{\bC}{\mathbb{C}}
\newcommand{\bN}{\mathbb{N}}
\newcommand{\bP}{\mathbb{P}}
\newcommand{\bQ}{\mathbb{Q}}
\newcommand{\bZ}{\mathbb{Z}}

\newcommand{\be}{\mathbf{e}}

\usepackage{todonotes}

\begin{document}

\title[A refined multiplication formula for cluster characters]{A refined multiplication formula for cluster characters}

\author{Bernhard Keller}
\address{B.~K.~: Universit\'e Paris de Paris\\
    UFR de Math\'ematiques\\
    CNRS\\
   Institut de Math\'ematiques de Jussieu--Paris Rive Gauche, IMJ-PRG \\   
    B\^{a}timent Sophie Germain\\
    75205 Paris Cedex 13\\
    France
}
\email{bernhard.keller@imj-prg.fr}
\urladdr{https://webusers.imj-prg.fr/~bernhard.keller/}

\author{Pierre-Guy Plamondon}
\address{P.-G.P.~: Laboratoire de Math\'ematiques de Versailles, UVSQ, CNRS, Universit\'e Paris-Saclay, Institut Universitaire de France (IUF)}
\email{pierre-guy.plamondon@uvsq.fr}
\urladdr{https://www.imo.universite-paris-saclay.fr/~plamondon/}

\author{Fan Qin}
\address{F.Q.: The School of Mathematical Sciences, Shanghai Jiao Tong University, 800 Dongchuan RD, Shanghai, 200240 China.}
\email{qin.fan.math@gmail.com}
\urladdr{https://sites.google.com/site/qinfanmath/}


\begin{abstract}
We obtain a multiplication formula for cluster characters on (stably) $2$-Calabi--Yau (Frobenius or) triangulated categories.  This formula generalizes those known for arbitrary pairs of objects and for Auslander--Reiten triangles.  As an application, we show that for cluster algebras of acyclic types, specialization of a cluster variable to~$1$ sends all cluster variables to elements of a cluster algebra of smaller rank.  We also obtain applications to the reduction of friezes of acyclic type.
\end{abstract}

\maketitle

\tableofcontents
\section{Introduction}

The additive categorification of cluster algebras has been an important tool in their study almost from their inception (see for instance the survey papers \cite{Kel09,Reiten10,Amiot11,Plamondon18}).  Such a categorification is given by a category~$\cC$ (usually triangulated or exact) and a \emph{cluster character} sending objects of~$\cC$ to Laurent polynomials in several variables so that suitable objects of~$\cC$ are sent to cluster variables in a cluster algebra.  The key property that a cluster character satisfies is a \emph{multiplication formula} which recovers the exchange relations in a cluster algebra.  Such formulas at various levels of generality have been obtained in \cite{CC06,CK08,CK06,GLS07,Palu08,FK09,DWZ09,Palu09,Plamondon09,XiaoXu2010,Xu2010,DG12,Rupel15,GLS18,CEFR21} and more.

The main result of this paper is a multiplication formula generalizing most previously known ones in the following context.  Let~$\cC$ be a small~$\Hom{}$-finite Krull--Schmidt~$2$-Calabi--Yau triangulated category over~$\bC$, together with a basic cluster tilting object~$T$ (definitions are recalled in Section~\ref{subs::recollections}).  Let
\[
 CC_T : Obj(\cC) \to \bZ[x_1^{\pm 1}, \ldots, x_n^{\pm 1}]
\]
be the corresponding cluster character (Definition~\ref{def::characters}).  For any objects~$L$ and~$M$ of~$\cC$, let~
\begin{displaymath}
	\beta_{L,M}:\Hom{\cC}(L, \Sigma M) \times \Hom{\cC}(M, \Sigma L) \longrightarrow k
\end{displaymath}
be the non-degenerate bifunctorial bilinear form conferring to~$\cC$ its~$2$-Calabi--Yau structure.  For an object~$Y$, let~$\Hom{\cC}(L,\Sigma M)_{\langle Y \rangle}$ be the set of those morphisms~$\varepsilon:L\to \Sigma M$ such that, if we have a triangle
\[
 M \xrightarrow{} Y' \xrightarrow{} L \xrightarrow{\varepsilon} \Sigma M,
\]
the objects~$Y$ and~$Y'$ have the same index (see Definition~\ref{defi::index}) and for each dimension vector~$\be$, the submodule Grassmannians~$\Gr{\be}(\Hom{\cC}(T,Y))$ and~$\Gr{\be}(\Hom{\cC}(T,Y'))$ have the same Euler characteristic.  It is easy to check that this set is invariant under multiplication by a non-zero scalar.  For a subset~$X$ of~$\Hom{\cC}(L,\Sigma M)$, let~$X_{\langle Y \rangle}$ be the intersection of~$X$ with~$\Hom{\cC}(L,\Sigma M)_{\langle Y \rangle}$. Let~$\cY_{L,M}$ be a set of representatives of equivalence classes for the equivalence relation defined by~$\Hom{\cC}(L,\Sigma M)_{\langle Y \rangle} = \Hom{\cC}(L,\Sigma M)_{\langle Y' \rangle}$.

Our main result is the following refined multiplication formula.

\begin{theostar}[\ref{theo::main}]
 Let~$\cC$ be a small~$\Hom{}$-finite Krull--Schmidt~$2$-Calabi--Yau triangulated category over~$\bC$ with constructible cones (see Section~\ref{subs::constructible}) together with a basic cluster tilting object~$T$.  Let~$L$ and~$M$ be objects of~$\cC$ such that~$\Hom{\cC}(L,\Sigma M)$ is non-zero.  Finally, let~$V$ be a non-zero vector subspace of~$\Hom{\cC}(L,\Sigma M)$. Then
 \[
  \chi(\bP V) CC_{T}(L)CC_{T}(M) = \sum_{Y\in \cY_{L,M}} \chi(\bP V_{\langle Y \rangle})CC_{T}(Y) + \sum_{Y\in \cY_{M,L}} \chi(\cR_{\langle Y\rangle})CC_{T}(Y),
 \]
 where~$\cR = \bP\Hom{\cC}(M,\Sigma L) \setminus \bP \Ker \beta_{L,M}(V,?)$.
\end{theostar}
If~$V$ is the full space, then this formula specializes to the one proved in~\cite[Theorem 1.1]{Palu09}.
Our main result also has a counterpart for exact categories.

\begin{theostar}[\ref{theo::main-exact}]
  Let $\cE$ be an $\Ext{}{}$-finite 2-Calabi--Yau Frobenius category with a cluster tilting object $T$.  Assume that the triangulated category $\underline{\cE}$ has constructible cones. Let $L$ and $M$ be two objects of $\cE$ such that~$\Ext{1}{\cE}(L,M)$ is non-zero, and let $V$ be a non-zero vector subspace of $\Ext{1}{\cE}(L,M)$.  Then
\begin{displaymath}
  \chi(\bP V)CC_{T}(L)CC_{T}(M) = \sum_{Y\in \cY_{L,M}} \chi(\bP V_{\langle Y \rangle})CC_{T}(Y) + \sum_{Y\in \cY_{M,L}} \chi(\cR_{\langle Y \rangle})CC_{T}(Y).
\end{displaymath}
\end{theostar}
This generalizes a result of~\cite{FK09}.  We expect that the refined multiplication formula generalizes to the setting of suitable extriangulated categories such as the Higgs category of~\cite{Wu-Higgs}, in which the classical multiplication formula can be proved~\cite{KellerWu}.

We apply our main results to the specialization of cluster variables to~$1$.  Let~$Q$ be finite quiver without loops or~$2$-cycles and let~$Q'$ be the quiver obtained from~$Q$ by removing a vertex~$i$.  Let~$\sigma$ be the specialization of~$x_i$ at~$1$.  In the case where~$Q$ is mutation-equivalent to an acyclic quiver, it was proved in~\cite{ADS} that the image of the cluster algebra~$\cA_Q$ by~$\sigma$ is contained in~$\cA_{Q'}\otimes_{\bZ}\bQ$.  Using our refined multiplication formula, we can improve on this result.

\begin{corostar}[\ref{coro::specializationAcyclic}]
 Assume that~$Q$ is mutation-equivalent to an acyclic quiver.  Then the image of the cluster algebra~$\cA_Q$ by~$\sigma$ is~$\cA_{Q'}$.  
\end{corostar}

More generally, we have the following results.

\begin{corostar}[\ref{coro::specialization-cluster-equal-upper}]
 Assume that the quiver~$Q$ admits a non-degenerate Jacobi-finite potential.  If the upper cluster algebra~$\cU_{Q'}$ is equal to the cluster algebra~$\cA_{Q'}$, then~$\sigma(\cA_Q) = \cA_{Q'}$.
\end{corostar}

\begin{corostar}[\ref{coro::specializationUpper}]
 Assume that the quiver~$Q$ admits a non-degenerate Jacobi-finite potential.  If the upper cluster algebra~$\cU_{Q'}$ is spanned by the cluster characters of objects of the associated generalized cluster category~$\cC$, then~$\sigma(\cU_Q) = \cU_{Q'}$.
\end{corostar}

Note that in the above results the variable that gets specialized to~$1$ is not frozen. 

Our formula also finds applications in the reduction of friezes.  A \emph{frieze} is ring morphism~$f:\cA_Q\to \bZ$ that sends all cluster variables to positive integers.  Friezes originated in work of Conway and Coxeter \cite{Coxeter71,ConwayCoxeter73}, but have been vastly generalized using cluster algebras, see for instance \cite{BaurMarsh09,ARS10,BFGST,MG19} and the survey paper \cite{MG15}.  Our result on friezes is the following.

\begin{corostar}[\ref{coro::reductionFriezes}]
 Let~$Q$ be an acyclic quiver without loops or~$2$-cycles, let~$Q'$ be the quiver obtained by removing the vertex~$i$ in~$Q$, and let~$\sigma:\cA_Q \to \cA_{Q'}$ be the specialization of~$x_i$ to~$1$.  Let~$f':\cA_{Q'}\to \bZ$ be a frieze.  Then there exists a unique frieze~$f:\cA_Q \to \bZ$ such that~$f'\circ \sigma = f$.  
\end{corostar}
The non-trivial part of the above result is the existence.

Finally, in Section~\ref{subs::AR}, we give a new proof of a multiplication formula for Auslander--Reiten triangles first obtained in~\cite{DG12}, and in Section~\ref{subs::DX10}, we obtain a formula reminiscent of the one stated in~\cite{DX10}.

\section{Refined multiplication formula: triangulated case}

\subsection{Recollections on $2$-Calabi--Yau triangulated categories}
The setting in which the multiplication formula holds is that of $\Hom{}$-finite, Krull-Schmidt, triangulated, $2$-Calabi--Yau categories with a cluster tilting object and constructible cones.  The aim of this section is to recall the main definitions and properties of this setting.

\subsubsection{$2$-Calabi--Yau categories}\label{subs::recollections}
Let $\cC$ be a small $\Hom{}$-finite triangulated category over a field $k$, with suspension functor $\Sigma$.  
\begin{definition}
The category $\cC$ is \emph{$2$-Calabi--Yau} if, for any objects $L$ and $M$ of $\cC$, it is equipped with a bilinear form
\begin{displaymath}
	\beta_{L,M}:\Hom{\cC}(L, \Sigma M) \times \Hom{\cC}(M, \Sigma L) \longrightarrow k
\end{displaymath}
which is non-degenerate and bifunctorial.  Here, bifunctorial means that if $L$, $M$, $N$ and $P$ are objects of $\cC$, and if $\varepsilon\in\Hom{\cC}(M, \Sigma N)$, $\eta\in \Hom{\cC}(N,\Sigma L)$, $\delta\in \Hom{\cC}(P, \Sigma M)$, $f\in\Hom{\cC}(L,M)$ and $g\in\Hom{\cC}(N,P)$, then
\begin{eqnarray*}
	\beta_{L,N}(\varepsilon\circ f, \eta) &=& \beta_{M,N}(\varepsilon, \Sigma f\circ \eta) \quad \textrm{and} \\
	\beta_{M,P}(\Sigma g \circ \varepsilon, \delta) &=& \beta_{M,N}(\varepsilon, \delta\circ g).
\end{eqnarray*}
\end{definition}

Equivalently, $\cC$ is $2$-Calabi--Yau if it  is equipped with an isomorphism of bifunctors
\begin{displaymath}
	\Hom{\cC}(L, \Sigma M) \longrightarrow D\Hom{\cC}(M, \Sigma L),
\end{displaymath}
where $D=\Hom{k}(?,k)$ is the usual duality for vector spaces.

\subsubsection{Cluster-tilting objects and associated cluster characters}
Let $\cC$ be a $\Hom{}$-finite $2$-Calabi--Yau triangulated category. 
\begin{definition}[\cite{BMRRT06}]
 An object $T$ of $\cC$ is a \emph{cluster-tilting object} if the following hold:
\begin{enumerate}
	\item $T$ is \emph{rigid}, that is, the space $\Hom{\cC}(T, \Sigma T)$ vanishes;
	\item for any object $X$ of $\cC$, if $\Hom{\cC}(T, \Sigma X)$ vanishes, then $X$ lies in $\add T$ (that is, $X$ is a direct factor of a direct sum of copies of $T$).
\end{enumerate}
\end{definition}
We will usually assume that $T$ is \emph{basic}, and write $T=T_1\oplus \ldots \oplus T_n$, where the $T_i$'s are pairwise non-isomorphic indecomposable objects.

\begin{examples}
 \begin{enumerate}
  \item The cluster categories of \cite{BMRRT06} are triangulated $\Hom{}$-finite $2$-Calabi--Yau categories with a cluster-tilting object.
  \item The generalized cluster categories of \cite{Amiot08} also have these properties.
  \item The stable categories of all the Frobenius categories of Example \ref{exam::Frobenius} also have these properties.
 \end{enumerate}

\end{examples}

Cluster-tilting objects are essential in the categorification of cluster algebras via triangulated categories.  This is done via \emph{cluster characters}, whose definition we recall in Definition \ref{def::characters}.

\begin{proposition}[\cite{KR07}]\label{prop::KellerReiten}
	Let $T$ be a basic cluster-tilting object of $\cC$.
	\begin{enumerate}
		\item The functor $F=\Hom{\cC}(T, \Sigma ?)$ induces an equivalence of categories
		\begin{displaymath}
			\cC/(T) \xrightarrow{F} \MOD \End{\cC}(T),
		\end{displaymath}
		where $(T)$ is the ideal of all morphisms factoring through an object of $\add T$.

		\item Any object $X$ of $\cC$ sits in a triangle
		\begin{displaymath}
			T_1^X \rightarrow T_0^X \rightarrow X \rightarrow \Sigma T_1^X,
		\end{displaymath}
		where $T_1^X$ and $T_0^X$ lie in $\add T$.
	\end{enumerate}
\end{proposition}

\begin{definition}[\cite{DK08}]\label{defi::index}
	Let $T$ be a basic cluster-tilting object of $\cC$.  The \emph{index} of an object $X$ of $\cC$ \emph{with respect to $T$} is the element of the Grothendieck group $K_0(\add T)$ defined by
	\begin{displaymath}
		\ind{T}X = [T_0^X] - [T_1^X],
	\end{displaymath}
where $T_0^X$ and $T_1^X$ are as in Proposition \ref{prop::KellerReiten}(2).
\end{definition}
Note that, while the triangle in Proposition \ref{prop::KellerReiten}(2) is not unique, the index of $X$ does not depend on the one we choose \cite[Lemma 2.1]{Palu08}.  Moreover, it was shown in \cite{Palu08} (in the proof of Lemma 1.3) that for any object $X$ of $\cC$, the value of $\ind{T}\Sigma X + \ind{T}X$ only depends on the dimension vector $e$ of $FX$.  We will denote this value by $\iota(e)$.  Note that~$\iota$ extends to a linear map defined on all of~$\bZ^n$.

Let us now assume that the field $k$ is the field $\bC$ of complex numbers.

\begin{definition}[\cite{CC06}\cite{Palu08}]\label{def::characters}
	Let $T$ be a basic cluster-tilting object of $\cC$.  The \emph{cluster character associated with $T$} is the map 
	\begin{displaymath}
		CC_{T}:Obj(\cC) \longrightarrow \bQ(x_1, \ldots, x_n)
	\end{displaymath}
	defined by 
	\begin{displaymath}
		CC_{T}(M) = x^{\ind{T}M}\sum_{e\in \bN^n} \chi \Big( \Gr{e}\big(FM   \big)  \Big)x^{-\iota(e)},
	\end{displaymath}
	where
	\begin{itemize}
		\item $n$ is the number of indecomposable direct factors of $T$ in a decomposition $T = \bigoplus_{i=1}^n T_i$;
		\item $x^a = x_1^{a_1} \cdots x_n^{a_n}$, for any $a = \sum_{i=1}^n a_i[T_i] \in K_0(\add T)$;
		\item $\chi$ is the Euler characteristic for topological spaces;
		\item $FM=\Hom{\cC}(T, \Sigma M)$ is considered as a right module over $\End{\cC}(T)$;
		\item for any module $R$, $\Gr{e}(R)$ is the submodule Grassmannian \cite[Section 2.3]{CC06}, a projective variety whose points parametrize the submodules of $R$ of dimension vector $e$;
		\item $\iota(e)$ is as defined below Definition~\ref{defi::index}.
	\end{itemize}
\end{definition}

\subsubsection{Constructible cones}\label{subs::constructible}

The coefficients in the multiplication formula are Euler characteristics of subsets of certain algebraic varieties.  For the formula to be well-defined, we must ensure that the Euler characteristics of these subsets are well-defined integers.  In \cite{Palu09}, this is done by proving that the subsets in question are constructible.  In order to do so, we need to assume that the category $\cC$ \emph{has constructible cones}.  Although we will not recall the definition of a category with consctructible cones (and simply refer to \cite[Section 1.3]{Palu09}), we will list the properties of such categories that we will need.

Let $\cC$ be a $\Hom{}$-finite triangulated category with a basic cluster-tilting object $T$.  Fix two objects $L$ and $M$.  For any object $Y$ of $\cC$, let $\Hom{\cC}(L, \Sigma M)_{\langle Y \rangle}$ be the subset of $\Hom{\cC}(L, \Sigma M)$ of all morphisms $\varepsilon$ such that if
\begin{displaymath}
  M \rightarrow Y' \rightarrow L \stackrel{\varepsilon}{\rightarrow} \Sigma M
\end{displaymath}
is a triangle, then 
\begin{itemize}
  \item $\ind{T} Y' = \ind{T} Y$, and
  \item for all dimension vectors $e$, we have  $\chi(\Gr{e}(FY))=\chi(\Gr{e}(FY'))$.
\end{itemize}
For any subset $V$ of $\Hom{\cC}(L, \Sigma M)$, let $V_{\langle Y \rangle}$ be the intersection of $V$ with $\Hom{\cC}(L, \Sigma M)_{\langle Y \rangle}$.

Note that the condition
\begin{displaymath}
  \Hom{\cC}(L, \Sigma M)_{\langle Y \rangle} = \Hom{\cC}(L, \Sigma M)_{\langle Y' \rangle}
\end{displaymath}
induces an equivalence relation on the set of objects of $\cC$.  Let $\cY_{L,M}$ be a set of representatives for this equivalence relation.

\begin{proposition}[Proposition 2.8 of \cite{Palu09}]
 If $\cC$ has constructible cones, then
 \begin{displaymath}
   \Hom{\cC}(L, \Sigma M) = \coprod_{Y\in \cY_{L,M}}\Hom{\cC}(L, \Sigma M)_{\langle Y \rangle}
 \end{displaymath}
is a partition of $\Hom{\cC}(L, \Sigma M)$ into a finite number of constructible subsets.
\end{proposition}


\begin{corollary}
If $\cC$ has constructible cones, and if $V$ is a constructible subset of $\Hom{\cC}(L, \Sigma M)$, then
\begin{displaymath}
 V = \coprod_{Y\in \cY_{L,M}}V_{\langle Y \rangle}
\end{displaymath}
is a decomposition of $V$ into a finite number of pairwise disjoint constructible subsets.
\end{corollary}

\begin{example}
 All the triangulated categories mentioned in this paper have constructible cones, thanks to these two facts proved in \cite[Sections 2.4-2.5]{Palu09}: stable categories of $\Hom{}$-finite Frobenius categories and the generalized cluster categories of \cite{Amiot08} have constructible cones.
\end{example}

\subsection{The refined multiplication formula}

This section is devoted to the proof of the refined multiplication formula.  The proof follows the lines of \cite{Palu09} and relies heavily on results obtained there.

For any vector space $E$ and for any subset $U$ which is stable by scalar multiplication, we denote by $\bP U$ the subset of the projective space $\bP E$ consisting of elements $[u]$ with $u\in U$, where~$[u]$ denotes the class of~$u$ in~$\bP E$.

\begin{theorem}\label{theo::main}
Let $\cC$ be a $\Hom{}$-finite Krull--Schmidt $2$-Calabi--Yau triangulated category over $\bC$ with constructible cones and admitting a basic cluster-tilting object $T$.  Let $L$ and $M$ be two objects of $\cC$, and let $V$ be a non-zero vector subspace of $\Hom{\cC}(L, \Sigma M)$.  Then the following equality holds:
\begin{displaymath}
	\chi(\bP V) CC_{T}(L)CC_{T}(M) = \sum_{Y\in \cY_{L,M}} \chi(\bP V_{\langle Y \rangle})CC_{T}(Y) + \sum_{Y\in \cY_{M,L}} \chi(\cR_{\langle Y\rangle})CC_{T}(Y),
\end{displaymath} 
where $\cR=\bP\Hom{\cC}(M,\Sigma L) \setminus \bP\Ker \beta_{L,M}(V,?)$. 
\end{theorem}

\begin{remark}
  If $V$ is the whole space $\Hom{\cC}(L, \Sigma M)$, then the formula recovers that of Y.~Palu \cite[Theorem 1.1]{Palu09}.
\end{remark}

We assume for the rest of this section that $\cC$ has constructible cones.  

The first step into proving the formula is by replacing $CC_{T}(L)$ and $CC_{T}(M)$ by their definitions in the left-hand side of the formula.  Doing this, we get
\begin{IEEEeqnarray*}{rCl}
	\IEEEeqnarraymulticol{3}{l}{\chi (\bP V)CC_{T}(L)CC_{T}(M)}\\ 
	 \ &=& \chi (\bP V)\Big( x^{\ind{T}L}\sum_{e}\chi\big( \Gr{e}(FL) \big) x^{-\iota(e)} \Big) \Big(x^{\ind{T}M} \sum_{f}\chi\big( \Gr{f}(FM) \big)x^{-\iota(f)}   \Big) \\
				&=& x^{\ind{T}(L\oplus M)}\sum_{e,f}\chi\big(\bP V \times  \Gr{e}(FL) \times \Gr{f}(FM) \big) x^{-\iota(e+f)}.
\end{IEEEeqnarray*}
We will refine this sum by replacing $\bP V \times \Gr{e}(FL)\times\Gr{f}(FM)$ by another constructible set with the same Euler characteristic.  Let us construct this set.

Define $W_{L,M}^{V}$ to be the subset of $\bP V \times \coprod_{d,g} \prod_{i=1}^n \Gr{g_i}(\bC^{d_i})$ consisting of pairs $([\varepsilon], E)$ where $E$ is a subrepresentation of $FY$, where $Y$ is the middle term of a triangle $\xymatrix{M \ar[r]^i & Y\ar[r]^p & L\ar[r]^\varepsilon & \Sigma M}$.

Furthermore, define
\begin{IEEEeqnarray*}{rCl}
	W_{L,M}^{V}(e,f,g) &=& \{ ([\varepsilon], E)\in W_{L,M}^{V} \ | \ \dimv E = g, \dimv Fp(E) = e, \dimv F i^{-1}(E) = f \} ; \\
	W_{L,M}^{V}(e,f) &=& \{ ([\varepsilon], E)\in W_{L,M}^{V} \ | \ \dimv Fp(E) = e, \dimv F i^{-1}(E) = f \} ; \\
	W_{L,M}^{V,Y}(e,f,g) &=& \{([\varepsilon], E)\in W_{L,M}^{V} \ | \ \varepsilon \in \bP V_{\langle Y \rangle}, \dimv E = g, \dimv Fp(E) = e, \dimv F i^{-1}(E) = f \};\\
	W_{L,M}^{V,Y}(e,f) &=& \{([\varepsilon], E)\in W_{L,M}^{V} \ | \ \varepsilon \in \bP V_{\langle Y \rangle}, \dimv Fp(E) = e, \dimv F i^{-1}(E) = f \}.
\end{IEEEeqnarray*}
Then $W_{L,M}^{V}$ and all the sets defined above are  finite disjoint unions of subsets of the form $W_{L,M}^{V,Y}(e,f,g)$.  Moreover, since we assumed that $\cC$ has constructible cones, then the results of Y.~Palu give us the following.

\begin{lemma}[Lemma 3.1 of \cite{Palu09}]\label{lemm::constructible}
	The sets $W_{L,M}^{V,Y}(e,f,g)$, $W_{L,M}^{V,Y}(e,f)$, $W_{L,M}^{V}(e,f)$, $W_{L,M}^{V}(e,f,g)$ and $W_{L,M}^{V}$ are constructible.
\end{lemma}
\demo{In \cite[Lemma 3.1]{Palu09}, it is shown that certain sets $W_{LM}^{Y}(e,f,g)$ are constructible.  Our sets $W_{L,M}^{V,Y}(e,f,g)$ are the intersection of these $W_{LM}^{Y}(e,f,g)$ with $\bP V \times \coprod_{d,g} \prod_{i=1}^n \Gr{g_i}(\bC^{d_i})$; thus they are constructible.  Since all the other sets are finite unions of sets of the form $W_{L,M}^{V,Y}(e,f,g)$, they must also be constructible.
}

Now, consider the constructible map
\begin{eqnarray*}
	\Psi_{L,M}(e,f): W_{L,M}^{V}(e,f) &\longrightarrow& \bP V \times \Gr{e}(FL) \times \Gr{f}(FM) \\
	([\varepsilon], E) &\longmapsto& ([\varepsilon], Fp(E), Fi^{-1}(E)).
\end{eqnarray*}
Let $L^V_1(e,f)$ be the image of this map; let $L_2^V(e,f)$ be the complement of the image.  Then $\chi(\bP V\times \Gr{e}(FL)\times \Gr{f}(FM)) = \chi(L_1^V(e,f)) + \chi(L_2^V(e,f))$.

Therefore our equation becomes

\begin{eqnarray*}
	(\star) \quad \chi(\bP V)CC_{T}(L)CC_{T}(M)  &=& x^{\ind{T}(L\oplus M)}\sum_{e,f}\chi\big( L_1^V(e,f) \big) x^{-\iota(e+f)} \\
	 && \quad + x^{\ind{T}(L\oplus M)}\sum_{e,f}\chi\big( L_2^V(e,f) \big) x^{-\iota(e+f)}.
\end{eqnarray*}
We will now study the two terms of the right-hand side of ($\star$).

\subsubsection*{The first term of the RHS of ($\star$)}

\begin{lemma}\label{lemm::term1}
We have an equality 
	\begin{displaymath}
		x^{\ind{T}(L\oplus M)}\sum_{e,f}\chi\big( L_1^V(e,f) \big) x^{-\iota(e+f)} = \sum_{Y\in\cY_{L,M}}\chi(\bP V_{\langle Y \rangle}) CC_{T}(Y).
	\end{displaymath}
\end{lemma}
\demo{ It is proved in \cite{CC06} (see also Section 3 of \cite{Palu09}) that the fibers of $\Psi_{L,M}(e,f)$ are affine spaces.  As a consequence, we have that $\chi(L^V_1(e,f)) = \chi(W_{L, M}^{V}(e,f))$. Thus
\begin{eqnarray*}
	x^{\ind{T}(L\oplus M)}\sum_{e,f}\chi\big( L_1^V(e,f) \big) x^{-\iota(e+f)} &=& x^{\ind{T}(L\oplus M)}\sum_{e,f}\chi\big( W_{L, M}^{V}(e,f) \big) x^{-\iota(e+f)} \\
&=& \sum_{e,f,g,\langle Y \rangle}\chi\big( W_{L, M}^{V,Y}(e,f,g) \big) x^{-\iota(e+f)+\ind{T}(L\oplus M)}.
\end{eqnarray*}
Now, by \cite[Lemma 5.1]{Palu08}, if $([\varepsilon], E)$ lies in $W_{L, M}^{V,Y}(e,f,g)$, then it implies that $\ind{T}(L\oplus M) - \iota(e+f) = \ind{T}(Y) - \iota(g)$.  

Moreover, for a fixed $g$, consider the map
\begin{displaymath}
	\coprod_{e,f} W_{L,M}^{V,Y}(e,f,g) \longrightarrow \bP V_{\langle Y \rangle}
\end{displaymath}
sending a pair $([\varepsilon], E)$ to $[\varepsilon]$.  This map is obviously surjective if the left-hand side is non-empty.  Moreover, the preimage of any $[\varepsilon']$ is isomorphic to $\{[\varepsilon']\} \times \Gr{g}(FY')$, where $Y'$ sits in a triangle $M\rightarrow Y'\rightarrow L \xrightarrow{\varepsilon'} \Sigma M$.  By definition of $\bP V_{\langle Y \rangle}$, the Euler characteristic of all the fibers is the same and is equal to $\chi(\Gr{g}(FY))$.  Thus
\begin{displaymath}
	\chi(\coprod_{e,f} W_{L,M}^{V,Y}(e,f,g)) = \chi(\Gr{g}(FY))\chi(\bP V_{\langle Y \rangle}).
\end{displaymath}

So the sum becomes
\begin{eqnarray*}
	... &=& \sum_{e,f,g,Y}\chi\big( W_{L, M}^{V,Y}(e,f,g) \big) x^{-\iota(g)+\ind{T}(Y)} \\
	&=& \sum_{g, Y}\chi\big( \coprod_{e,f} W_{L, M}^{V,Y}(e,f,g) \big) x^{-\iota(g)+\ind{T}(Y)} \\
	&=& \sum_{g, Y}\chi(\Gr{g}(FY))\chi(\bP V_{\langle Y \rangle}) x^{-\iota(g)+\ind{T}(Y)} \\
	&=& \sum_{Y\in\cY_{L,M}}\chi(\bP V_{\langle Y \rangle}) CC_{T}(Y).
\end{eqnarray*}
This finishes the proof of the lemma.  
}

\subsubsection*{The second term of the RHS of ($\star$)}

\begin{lemma}\label{lemm::term2}
	We have an equality
		\begin{displaymath}
			x^{\ind{T}(L\oplus M)}\sum_{e,f}\chi\big( L_2^V(e,f) \big) x^{-\iota(e+f)} = \sum_{Y\in\cY_{M,L}}\chi(\cR_{\langle Y \rangle}) CC_{T}(Y).
		\end{displaymath}
\end{lemma}
\demo{  Recall that \[\cR = \{ [\eta]\in \bP \Hom{\cC}(M, \Sigma L) \ | \ \exists \varepsilon\in V \textrm{ with } \beta_{L,M}(\varepsilon,\eta) \neq 0 \}.\] 

Define $W_{M,L}^{\bP\Hom{\cC}(L, \Sigma M), Y}(f,e,g)$ as before Lemma \ref{lemm::constructible} and let $W_{M,L}^{\cR, Y}(f,e,g)$ be the constructible subset of all pairs $([\eta], E)$ with $\eta \in \cR$.  For fixed $e$, $f$ and $g$, let $C_{L,M}^{\cR, Y}(e,f,g)$ be the subset of $L_2^{V}(e,f)\times W_{M,L}^{\cR, Y}(f,e,g)$ consisting of pairs $\big( ([\varepsilon], R, S), ([\eta], E) \big)$ such that $\beta_{L,M}(\varepsilon,\eta) \neq 0$, $Fi^{-1}(E) = S$ and $Fp(E) = R$.  Finally, let $C_{L,M}^{\cR}(e,f) = \coprod_{g, Y\in\cY_{M,L}}C_{L,M}^{\cR, Y}(e,f,g)$.

Consider the two projections
\begin{eqnarray*}
	C_{L,M}^{\cR}(e,f) &\xrightarrow{p_1}& L_2^{V}(e,f) \\
	C_{L,M}^{\cR, Y}(e,f,g) &\xrightarrow{p_2}& W_{M,L}^{\cR, Y}(f,e,g).
\end{eqnarray*}

By \cite[Proposition 3.3]{Palu09}, $p_1$ and $p_2$ are surjective.  Moreover, by \cite[Proposition 3.4]{Palu09}, the fibers of $p_1$ are extensions of affine spaces, and those of $p_2$ are affine spaces.

Therefore $\chi(C_{L,M}^{\cR}(e,f)) = \chi( L_2^{V}(e,f))$ and $\chi(C_{L,M}^{\cR, Y}(e,f,g)) = \chi( W_{M,L}^{\cR, Y}(f,e,g))$.  Thus the left-hand side in the statement is equal to
\begin{eqnarray*}
	... &=& x^{\ind{T}(L\oplus M)}\sum_{e,f}\chi(L_2^V(e,f))x^{-\iota(e+f)} \\
	&=& \sum_{e,f}\chi(C^\cR_{L,M}(e,f))x^{\ind{T}(L\oplus M)-\iota(e+f)} \\
	&=& \sum_{e,f,g,Y}\chi(C^{\cR,Y}_{L,M}(e,f,g))x^{\ind{T}(L\oplus M)-\iota(e+f)} \\
	&=& \sum_{e,f,g,Y}\chi(W^{\cR,Y}_{M,L}(f,e,g))x^{\ind{T}(L\oplus M)-\iota(e+f)}.
\end{eqnarray*}
Again, by \cite[Lemma 5.1]{Palu08}, we have that if $([\varepsilon], E)$ lies in $W^{\cR, Y}_{M,L}(f,e,g)$, then $\ind{T}(L\oplus M) - \iota(e+f) = \ind{T}(Y) - \iota(g)$.  Moreover, the map
\begin{displaymath}
  \coprod_{e,f}W^{\cR, Y}_{M,L}(f,e,g) \longrightarrow \cR_{\langle Y \rangle}
\end{displaymath}
sending $([\varepsilon],E)$ to $[\varepsilon]$ is surjective (if the left-hand side is non-empty), and its fibers have the form $\{[\varepsilon']\} \times \Gr{g}(FY')$, where $Y'$ sits in a triangle $L\rightarrow Y' \rightarrow M\xrightarrow{\varepsilon'} \Sigma L$.  Thus
\begin{displaymath}
  \chi(\coprod_{e,f}W^{\cR, Y}_{M,L}(f,e,g)) = \chi(\cR_{\langle Y \rangle})\chi(\Gr{g}(FY)).
\end{displaymath}
Therefore the above sequence of equalities continues:
\begin{eqnarray*}
 \ldots &=& \sum_{e,f,g,Y}\chi(W^{\cR,Y}_{M,L}(f,e,g))x^{\ind{T}Y-\iota(g)} \\
	&=& \sum_{g,Y}\chi(\coprod_{e,f}W^{\cR,Y}_{M,L}(f,e,g))x^{\ind{T}Y-\iota(g)} \\
	&=& \sum_{g,Y}\chi(\cR_{\langle Y \rangle})\chi(Gr{g}(FY))x^{\ind{T}Y-\iota(g)} \\
	&=& \sum_{Y\in \cY_{M,L}} \chi(\cR_{\langle Y \rangle})CC_{T}(Y).
\end{eqnarray*}
This finishes the proof.

}

Theorem \ref{theo::main} then follows directly from Lemma \ref{lemm::term1} and Lemma \ref{lemm::term2}.

\section{Refined multiplication formula: Frobenius case} 
We will now follow the ideas of \cite{FK09} (see also \cite[Section 4]{Palu09}).  The main difference will be that our Frobenius categories can be $\Hom{}$-infinite; we will only assume that they are $\Ext{}{}$-finite.  We will also assume that they are Krull--Schmidt, and that their stable categories have constructible cones.  Since the proofs are very similar to the ones in the triangulated case, in this section we only provide a detailed outline for the Frobenius case.
\subsection{Recollections on Frobenius categories}
\subsubsection{2-Calabi--Yau Frobenius categories}
A \emph{Frobenius category} is an exact category $\cE$ in the sense of Quillen with enough projectives and enough injectives, in which projectives and injectives coincide. It is \emph{$\Ext{}{}$-finite} if for any objects $X$ and $Y$ of $\cE$, the space $\Ext{1}{\cE}(X,Y)$ is finite-dimensional.

It was proved in \cite[Theorem 9.4]{Happel87} that if $\cE$ is a Frobenius category, then its stable category $\underline{\cE}$ is triangulated ($\underline{\cE}$ is the quotient of $\cE$ by the ideal of all morphisms factoring through a projective-injective object). Note that if $\cE$ is $\Ext{}{}$-finite, then $\underline{\cE}$ is $\Hom{}$-finite.

\begin{definition}[Section 2.7 of \cite{FK09}]
  An $\Ext{}{}$-finite Frobenius category is \emph{2-Calabi--Yau} if its stable category is 2-Calabi--Yau as a triangulated category.
\end{definition}

\subsubsection{Cluster-tilting objects}
Let $\cE$ be an $\Ext{}{}$-finite Krull--Schmidt 2-Calabi--Yau Frobenius category.

\begin{definition}[Section 2.7 of \cite{FK09}]
An object $T$ of $\cE$ is a \emph{cluster-tilting object} if
\begin{enumerate}
 \item $T$ is rigid, that is, the space $\Ext{1}{\cE}(T,T)$ vanishes;
 \item for any object $X$ of $\cE$, if $\Ext{1}{\cE}(T,X)$ vanishes, then $X\in \add T$; and
 \item each object $X$ of $\cE$ admits a right $\add T$-approximation $T^X\to X$ and a left $\add T$-approximation $X\to T_X$ (in other words, the functors $\Hom{\cE}(X,?)|_{\add T}$ and $\Hom{\cE}(?,X)|_{(\add T)^{op}}$ are finitely generated).
\end{enumerate}
\end{definition}
Note that if $T$ is a cluster-tilting object, then every indecomposable projective-injective object is isomorphic to a direct summand of $T$.

\begin{examples}\label{exam::Frobenius}
 \begin{enumerate}
  \item The module category of a preprojective algebra of Dynkin type is a $\Hom{}$-finite stably $2$-Calabi--Yau Frobenius category with a cluster tilting object.  It was used in \cite{GLS06} in the categorification of cluster algebras.  Its stable category has constructible cones by \cite[Section 2.4]{Palu09}.
  
  \item More generally, subcategories $\cC_w$ of modules over preprojective algebras were were used in \cite{BIRSc,GLS10} to category cluster algebras.  These categories have the same properties as those of the previous example.
  
  \item   Let $0<k<n$ be integers, and put $\hat R=\bC[[x,y]]/(x^k-y^{n-k})$.  The group $G=\langle \zeta\rangle$ of $n$-th roots of unity acts on $\hat R$ by $\zeta.x=\zeta x$ and $\zeta.y=\zeta^{-1}y$. Then the category $\cE= CM_G(\hat R)$ of $G$-equivariant Cohen--Macaulay $\hat R$-modules is a (not necessarily $\Hom{}$-finite) $\Ext{}{}$-finite Frobenius category with a cluster tilting object.  Its stable category has constructible cones, since it is equivalent to categories from the previous example. This was used in \cite{JensenKingSu} to give a categorification of the cluster algebra structure of the homogeneous coordinate ring $\bC[G_{k,n}]$ of the Grassmannian $G_{k,n}$.
 \end{enumerate}
\end{examples}

Let $T$ be a basic cluster-tilting object of $\cE$.  Write $T=T_1\oplus\ldots\oplus T_n$, where each $T_i$ is indecomposable, and let $C=\End{\cE}(T)$.  Since $T$ is cluster-tilting, we have a functor
\begin{displaymath}
 F=\Hom{\cE}(T,?): \cE \longrightarrow \fgmod C,
\end{displaymath}
where $\fgmod C$ is the category of finitely generated right $C$-modules.
For any finitely generated $C$-modules $L$ and $N$ such that $N$ is finite-dimensional, define
\begin{displaymath}
  \langle L,N \rangle_\tau = \dim \Hom{C}(L,N)-\dim \Ext{1}{C}(L,N),
\end{displaymath}
\begin{displaymath}
  \langle L,N \rangle_3 = \sum_{i=0}^3 (-1)^i \dim \Ext{i}{C}(L,N).
\end{displaymath}
Note that these expressions are well-defined integers, since $\Ext{i}{C}(L,N)$ is finite-dimensional because $\cE$ is $\Ext{}{}$-finite, and since $\Hom{C}(L,N)$ is finite-dimensional because $L$ is finitely generated and $N$ is finite-dimensional.

Finally, let $\underline{C} = \End{\underline{\cE}}(T)$.  As in \cite[Section 4]{KR07} and \cite[Section 3]{FK09}, we view $\underline{C}$-modules as $C$-modules with no composition factors isomorphic to the simple modules corresponding to the projective-injective direct summands of $T$.

\begin{proposition}[Proposition 3.2 of \cite{FK09}]
  If $L$ and $N$ are finite-dimensional $\underline{C}$-modules of the same dimension vector, then for any finite-dimensional $C$-module $Y$, we have that
\begin{displaymath}
  \langle L,Y \rangle_3 = \langle N,Y \rangle_3.
\end{displaymath}
\end{proposition}
In view of this proposition, if $e$ is a dimension vector, we can write $\langle e,Y \rangle_3$ for the value of $\langle L,Y \rangle_3$ for any $\underline{C}$-module $L$ of dimension vector $e$.

\begin{definition}[\cite{FK09}]
  The \emph{cluster character associated with $T$} is the map
\begin{displaymath}
  CC_{T}:Obj(\cE) \longrightarrow \bQ(x_1, \ldots, x_n)
\end{displaymath}
defined by
\begin{displaymath}
  CC_{T}(M) = \prod_{i=1}^n x_i^{\langle FM, S_i \rangle_\tau} \sum_e \chi\big( \Gr{e}( \Ext{1}{\cE}(T,M) ) \big) \prod_{i=1}^n x_i^{\langle e, S_i \rangle_3}.
\end{displaymath}
\end{definition}

\subsection{The formula}

\begin{theorem}\label{theo::main-exact}
  Let $\cE$ be a $\Hom{}$-finite 2-Calabi--Yau Frobenius category with a cluster tilting object $T$.  Assume that the triangulated category $\underline{\cE}$ has constructible cones. Let $L$ and $M$ be two objects of $\cE$, and let $V$ be a vector subspace of $\Ext{1}{\cE}(L,M)$.  Then
\begin{displaymath}
  \chi(\bP V)CC_{T}(L)CC_{T}(M) = \sum_{Y\in \cY_{L,M}} \chi(\bP V_{\langle Y \rangle})CC_{T}(Y) + \sum_{Y\in \cY_{M,L}} \chi(\cR_{\langle Y \rangle})CC_{T}(Y).
\end{displaymath}
\end{theorem}
The proof follows the lines of that of \cite[Theorem 4.1]{Palu09}; it is similar to that of Theorem \ref{theo::main}, but uses \cite[Lemma 3.4]{FK09} instead of \cite[Lemma 5.1]{Palu08}.

\section{Applications}

\subsection{Specialization of cluster variables in cluster algebras}

Let~$\cC$ be a $\Hom{}$-finite 2-Calabi--Yau triangulated category with a basic cluster tilting object $T=\bigoplus_{i=1}^n T_i$.  Following~\cite{CLS05}, let the \emph{Caldero-Chapoton algebra}~$\cA_{\cC}$ be the subring of the ring~$\bZ[x_1^{\pm 1}, \ldots, x_n^{\pm 1}]$ generated by the set of all~$CC_{T}(X)$, as~$X$ spans all objects of~$\cC$. 

Motivated by the reduction of friezes (see Section~\ref{subs::friezes}) and the study of morphisms of rooted cluster algebras of~\cite{ADS}, we wish to study the algebra obtained from~$\cA_{\cC}$ by specializing~$x_n$ to~$1$.  To fix notation, let
\[
 \sigma:\bZ[x_1^{\pm 1}, \ldots, x_n^{\pm 1}] \to \bZ[x_1^{\pm 1}, \ldots, x_{n-1}^{\pm 1}]
\]
be the morphism sending each of~$x_1, \ldots, x_{n-1}$ to itself and sending~$x_n$ to~$1$.

The main result of this section is stated in terms of Calabi--Yau reduction: it was proved in~\cite{IY08} that the category
\[
 \cC' = \left( \Sigma^{-1}T_n \right)^{\perp} / \left(T_n\right)
\]
is a $\Hom{}$-finite 2-Calabi--Yau triangulated category with a basic cluster tilting object $T'$, where~$T'$ is the image of~$T$ under the projection functor.

\begin{theorem}\label{theo::cc-specialization}
 Let~$\cC$ be a $\Hom{}$-finite 2-Calabi--Yau triangulated category with constructible cones and a basic cluster tilting object $T=\bigoplus_{i=1}^n T_i$.  Then~$\sigma(\cA_{\cC}) = \cA_{\cC'}$.
\end{theorem}

\begin{remark}
 The proof of \cite[Theorem 6.13]{ADS} shows that~$\sigma(\cA_{\cC}) \subseteq \cA_{\cC'}\otimes_{\bZ} \bQ$.  It uses the multiplication formula of Palu~\cite{Palu09}.  Our proof of Theorem~\ref{theo::cc-specialization} follows the same lines using our refined multiplication formula (Theorem~\ref{theo::main}) instead.  We nonetheless include the complete argument below.
\end{remark}

\begin{proof} 
 (of Theorem~\ref{theo::cc-specialization}).  It suffices to prove that~$\sigma\left(CC_T(X)\right) \in \cA_{\cC'}$ for all objects~$X$ of~$\cC$. The proof is by induction on the dimension of the space~$\Hom{\cC}(T_n, \Sigma X)$.  Assume first that $\dim \Hom{\cC}(T_n, \Sigma X) = 0$.  Then~$X\in (\Sigma^{-1}T_n)^{\perp}$. Denote by~$\pi$ the projection
 \[
  \pi: \left( \Sigma^{-1}T_n \right)^{\perp} \to \cC'.
 \]
 Then~$\sigma\left(CC_T(X)\right) = CC_{T'}(\pi X) \in \cA_{\cC'}$. Note that this implies that~$\cA_{\cC'} \subseteq \sigma(\cA_{\cC})$, since all objects of~$\cC'$ have the form~$\pi(X)$ with~$X$ an object of~$\cC$.
 
 \smallskip

 Assume now that~$\dim \Hom{\cC}(T_n, \Sigma X) = d > 0$.  Choose any non-split triangle
 \[
  X\xrightarrow{} E \xrightarrow{} T_n \xrightarrow{\xi} \Sigma X
 \]
 and let~$V$ be the span on~$\xi$ in~$\Hom{\cC}(T_n, \Sigma X)$.  Applying Theorem~\ref{theo::main}, we get that
 \[
  CC_T(X)CC_T(T_n) = CC_T(E) + \sum_{Y\in \cY_{X, T_n}} \chi\big( \cR_{\langle Y \rangle} \big) CC_T(Y).
 \]
 Since~$CC_T(T_n)=x_n$, applying the specialization~$\sigma$ to the left-hand side yields~$\sigma\left(CC_T(X)\right)$.  Since all~$\chi\big( \cR_{\langle Y \rangle} \big)$ are integers, it thus suffices to prove that~$CC_T(E)$ and all~$CC_T(Y)$ on the right-hand side are in~$\cA_{\cC'}$.  We do this by showing that the dimensions of~$\Hom{\cC}(T_n, \Sigma E)$ and~$\Hom{\cC}(T_n, \Sigma Y)$ are strictly smaller than~$d$ and by applying induction.
 
 To see this, first apply the functor~$\Hom{\cC}(T_n, ?)$ to the triangle defined by~$\xi$.  We obtain an exact sequence
 \[
  (T_n, T_n) \xrightarrow{\xi_*} (T_n, \Sigma X) \xrightarrow{f} (T_n, \Sigma E) \xrightarrow{} (T_n, \Sigma T_n),
 \]
 where we write~$(U,V)$ instead of~$\Hom{\cC}(U,V)$ to save space.  Since~$T_n$ is rigid,~$(T_n, \Sigma T_n)$ vanishes, so~$f$ is surjective; thus,
 \[
  (T_n, \Sigma E) \cong (T_n, \Sigma X)/\xi_*\big((T_n,T_n)\big).
 \]
 Lastly,~$\xi_*\big((T_n,T_n)\big)$ is non-zero, since it contains~$\xi_*(id_{T_n}) = \xi$.  Therefore,
 \[
  \dim (T_n, \Sigma E) < \dim (T_n, \Sigma X) = d,
 \]
 and by induction,~$\sigma\left(CC_T(E)\right) \in \cA_{\cC'}$.
 
 Now let~$Y \in \cY_{T_n, X}$.  By definition, there exists a non-split triangle
 \[
  T_n \xrightarrow{} Y \xrightarrow{} X \xrightarrow{\delta} \Sigma T_n.
 \]
 Applying the functor~$\Hom{\cC}(?, T)$ and repeating the above argument, we get that
 \[
  \dim (Y, \Sigma T_n) < \dim (X, \Sigma T_n) = d.
 \]
 By the~$2$-Calabi--Yau property,~$\dim (Y, \Sigma T_n) = \dim (T_n, \Sigma Y)$.  Thus, by induction, we also have that~$\sigma\left(CC_T(Y)\right) \in \cA_{\cC'}$.  This finishes the proof. 
\end{proof}

Theorem~\ref{theo::cc-specialization} has an interesting application to cluster algebras.

\begin{corollary}\label{coro::specialization-cluster-equal-upper}
 Let~$Q$ be a quiver without loops or~$2$-cycles,~$i$ be a vertex of~$Q$ and~$Q'$ be the quiver obtained from~$Q$ by removing the vertex~$i$.  Assume that there exists a non-degenerate potential~$W$ such that the generalized cluster category~$\cC_{Q,W}$ is~$\Hom{}$-finite.
 
 If the cluster algebra~$\cA_{Q'}$ is equal to its upper cluster algebra~$\cU_{Q'}$ (see \cite{BFZ05} for details),  then the specialization~$\sigma$ sending~$x_i$ to~$1$ satisfies~$\sigma(\cA_{Q}) = \cA_{Q'}$.
\end{corollary}
\begin{proof}
  Let~$\cC = \cC_{Q,W}$ and~$\cC' = \cC_{Q',W'}$, where~$W'$ is the potential obtained by removing the terms of~$W$ involving the vertex~$i$.  We know from~\cite{Palu08} that~$\cA_Q \subset \cA_{\cC}$, and from~\cite[Corollary 4.14]{Plamondon10} that~$\cA_{\cC} \subset \cU_Q$.  The same is true if we replace~$Q$ with~$Q'$; thus, by our assumption, $\cA_{Q'} = \cA_{\cC'} = \cU_{Q'}$.  Applying~$\sigma$, we get that~$\cA_{Q'} \subset \sigma(\cA_{Q}) \subset \sigma(\cA_{\cC})$, and this last set is~$\cA_{\cC'}$ by Theorem~\ref{theo::cc-specialization}.  This finishes the proof.
\end{proof}

\begin{corollary}\label{coro::specializationAcyclic}
 If~$Q$ is mutation-equivalent to an acyclic quiver, then we have that~$\sigma(\cA_Q) = \cA_{Q'}$.
\end{corollary}

\begin{corollary}\label{coro::specializationUpper}
 Assume that the quiver~$Q$ admits a non-degenerate Jacobi-finite potential, and let~$\cC$ be its generalized cluster category.  If~the upper cluster algebra $\cU_{Q'}$ is spanned the cluster characters of some objects in $\cC$, then we have that~$\sigma(\cU_{Q}) = \cU_{Q'}$.
\end{corollary}
\begin{proof}
Notice that we have $\cA_{\cC'}\subset \cU_{Q'}$ and $\cA_{\cC}\subset \cU_{Q}$ by the universal Laurent property of cluster characters. Since $\cU_{Q'}$ is spanned by some cluster characters, we have $\cU_{Q'}\subset \cA_{\cC'}$. Consequently, Theorem \ref{theo::cc-specialization} implies $\cU_{Q'}=\cA_{\cC'}=\sigma (\cA_{\cC})\subset \sigma (\cU_{Q})$.

Notice that every cluster for $Q'$ (see \cite{BFZ05}) is the image of some cluster for $Q$ under $\sigma$. So we have $ \sigma (\cU_{Q})\subset \cU_{Q'}$. The desired claim follows.
\end{proof}
Many upper cluster algebras are known to possess a generic basis, see \cite{GLS-generic,plamondon2013generic,qin2019bases}. They satisfy the assumption in Corollary \ref{coro::specializationUpper}.

\subsection{Reduction of friezes}\label{subs::friezes}

Let~$Q$ be a quiver without loops or~$2$-cycles.  A \emph{frieze} is a morphism of rings~$f:\cA_Q \to \bZ$ sending every cluster variable of~$\cA_Q$ to a positive integer.  This definition generalizes the originial one of Conway and Coxeter~\cite{ConwayCoxeter}, and has been an area of active interest in recent years.  

In~\cite[Section 5]{BFGST}, an operation of reduction on friezes is considered.  The purpose of this section is to show that this reduction operation can be ``reversed'' by adding a~$1$ in a frieze.  

\begin{corollary}\label{coro::reductionFriezes}
 Let~$Q$ be an acyclic quiver without loops or~$2$-cycles, and let~$Q'$ be the quiver obtained by removing the vertex~$i$ in~$Q$, and let~$\sigma:\cA_Q \to \cA_{Q'}$ be the specialization of~$x_i$ to~$1$ (this is well-defined thanks to Corollary~\ref{coro::specializationAcyclic}).  Let~$f':\cA_{Q'}\to \bZ$ be a frieze.  Then there exists a unique frieze~$f:\cA_Q \to \bZ$ such that~$f'\circ \sigma = f$.  
\end{corollary}
\begin{proof}
 If~$f$ exists, then it is unique, since it is determined by its action on the initial cluster variables of~$\cA_Q$.  Let us prove that such an~$f$ exists for any frieze~$f'$.  By Corollary~\ref{coro::specializationAcyclic},~$f=f'\circ\sigma$ is a well-defined morphism of rings from~$\cA_Q$ to~$\bZ$.  We only need to check that all cluster variables of~$\cA_Q$ are sent to positive values by~$f$; this follows from the positivity theorem~\cite{LeeSchiffler}: any cluster variable of~$\cA_Q$ is a Laurent polynomial with nonnegative coefficients in the initial cluster variables, and these are sent to positive values by~$f$. 
\end{proof}

\begin{remark}
 Corollary~\ref{coro::reductionFriezes} can be deduced for friezes where~$Q$ is of type $A_n$ from the results of~\cite{ConwayCoxeter}, and was shown to be true in types~$A_n,D_n$ and~$E_6$ in \cite[Section 5]{BFGST}, where it was also observed to be true for all known friezes of types~$E_7$ and~$E_8$ by a direct check.  The total number of possible friezes in these types is still unknown and was conjectured in~\cite{FP16}.  
\end{remark}

\subsection{A formula for Auslander--Reiten triangles}\label{subs::AR}

In this section, we will show how Theorem \ref{theo::main} allows for a new proof of the following formula of S.~Dominguez and C.~Geiss when $\cC$ has constructible cones.

\begin{theorem}[Theorem 1 of \cite{DG12}]
  Let $\cC$ be a $\Hom{}$-finite 2-Calabi--Yau category with constructible cones and a cluster tilting object $T$.  Let $Z$ be an indecomposable object of $\cC$, and assume that it sits in an Auslander--Reiten triangle
\begin{displaymath}
  \Sigma Z \xrightarrow{\alpha} Y \xrightarrow{\beta} Z \xrightarrow{\varepsilon} \Sigma^2 Z.
\end{displaymath}
Then
\begin{displaymath}
  CC_{T}(Z)CC_{T}(\Sigma Z) = CC_{T}(Y) + 1.
\end{displaymath}
\end{theorem}

We will give a proof of this theorem under the additionnal assumption that \emph{$\cC$ has constructible cones}.

Let $V$ be the one-dimensional subspace of $\Hom{\cC}(Z, \Sigma^2 Z)$ generated by $\varepsilon$.  Applying Theorem \ref{theo::main}, we get
\begin{displaymath}
  CC_{T}(Z)CC_{T}(\Sigma Z) = CC_{T}(Y) + \sum_{E\in \cY_{\Sigma Z, Z}}\chi(\cR_{\langle E \rangle})CC_{T}(E).
\end{displaymath}
Here $\cR = \{[\eta]\in\bP \Hom{\cC}(\Sigma Z, \Sigma Z) \ | \ \beta_{Z, \Sigma Z}(\varepsilon, \eta) \neq 0 \}$.  Let us show that  $[\eta]$ lies in $\cR$ if and only if $\eta$ is an isomorphism.

Since $Z$ is indecomposable, $\Hom{\cC}(\Sigma Z, \Sigma Z)$ is a local ring.  Hence $\eta$ is an isomorphism if and only if it does not lie in the radical of $\Hom{\cC}(\Sigma Z, \Sigma Z)$.  Thus we need to show that $\eta$ lies in the radical of $\Hom{\cC}(\Sigma Z, \Sigma Z)$ if and only if $\beta_{Z, \Sigma Z}(\varepsilon, \eta)=0$.

Assume that $\eta$ is in the radical (assume $\eta \neq 0$; the case $\eta=0$ is trivial).  Then it is not an isomorphism, and since $Z$ is indecomposable, it is not a retraction.   Then, by definition of an Auslander--Reiten triangle, we must have that there exists a morphism $f:Z\rightarrow Y$ such that $\Sigma^{-1}\eta = \beta f$.  But then
\begin{eqnarray*}
  \beta_{Z, \Sigma Z}(\varepsilon, \eta) & = & \beta_{Z, \Sigma Z}(\varepsilon, \Sigma \beta \Sigma f) \\
                                         & = & \beta_{Y, \Sigma Z}(\varepsilon\beta, \Sigma f) \\
                                         & = & \beta_{Y, \Sigma Z}(0, \Sigma f) \\
                                         & = & 0.
\end{eqnarray*} 

Assume next that $\eta$ is an isomorphism.  Then $\eta$ and~$\rad \Hom{\cC}(\Sigma Z, \Sigma Z)$ generate $\Hom{\cC}(\Sigma Z, \Sigma Z)$ as a vector space. If $\beta_{Z, \Sigma Z}(\varepsilon, \eta)$ were to vanish, it would thus vanish for any $\eta'$ in $\Hom{\cC}(\Sigma Z, \Sigma Z)$, contradicting the fact that $\beta_{Z, \Sigma Z}$ is non-degenerate.  Thus $\beta_{Z, \Sigma Z}(\varepsilon, \eta)\neq 0$.

This proves that $\cR$ is the set of $[\eta]$, with $\eta$ an isomorphism.  But then $\cR = \cR_{\langle 0 \rangle}$ (since the middle term of a triangle associated with an isomorphism is 0).  Moreover, $\cR = \bP \Hom{\cC}(\Sigma Z, \Sigma Z) \setminus \bP\! \rad\Hom{\cC}(\Sigma Z, \Sigma Z)$ is an affine space, since $\rad\Hom{\cC}(\Sigma Z, \Sigma Z)$ is a hyperplane in $\Hom{\cC}(\Sigma Z, \Sigma Z)$.  Thus $\chi(\cR)=1$.

Therefore
\begin{eqnarray*}
  CC_{T}(Z)CC_{T}(\Sigma Z) &=& CC_{T}(Y) + \sum_{E\in \cY_{Z, \Sigma Z}}\chi(\cR_{\langle E \rangle})CC_{T}(E) \\
                    &=& CC_{T}(Y) + \chi(\cR_{\langle 0 \rangle})CC_{T}(0) \\
                    &=& CC_{T}(Y) + 1.
\end{eqnarray*}
This finishes the proof.

\subsection{Another restricted formula}\label{subs::DX10}
Theorem \ref{theo::main} allows us to obtain (always assuming contructibility of cones) the following formula, reminiscent of the one stated in \cite{DX10}.  For two objects $L$ and $M$ of $\cC$, let $(T)(L,M)$ be the space of morphisms from $L$ to $M$ factoring through an object of $\add T$.

\begin{proposition}\label{prop::DX10}
 Under the hypotheses of Theorem \ref{theo::main}, we have that
 \begin{IEEEeqnarray*}{rCl}
  \chi\big(\bP(T)(L,\Sigma M)\big)CC_{T}(L)CC_{T}(M)  &=& \sum_{Y\in \cY_{L,M}}\chi\big(\bP(T)(L,\Sigma M)_{\langle Y \rangle}\big) CC_{T}(Y) \\
     && \hspace{-3cm} +\: \sum_{Y\in \cY_{M,L}}\chi\big(\bP\Hom{\cC}(M,\Sigma L)_{\langle Y\rangle} \setminus \bP(T)(M,\Sigma L)_{\langle Y\rangle} \big) CC_{T}(Y).
 \end{IEEEeqnarray*}
\end{proposition}
\demo{This follows from Theorem \ref{theo::main} by taking $V=(T)(L,\Sigma M)$. To see this, we only need to prove that $\Ker \beta_{L,M}(V,?)=(T)(M,\Sigma L)$.

Notice first that $(T)(M,\Sigma L)$ is contained in $\Ker \beta_{L,M}(V,?)$; indeed, if $f\in V$ and $g\in (T)(M,\Sigma L)$, then $\Sigma g\circ f=0$ (since $T$ is rigid), so 
\begin{displaymath}
 \beta_{L,M}(f,g) = \beta_{L,\Sigma L}(\Sigma g\circ f, id_{\Sigma L}) = 0.
\end{displaymath}

Moreover, $\dim \Ker \beta_{L,M}(V,?) = \dim \Hom{\cC}(L,\Sigma M)/V$, and this last vector space is isomorphic to the dual of $(T)(M, \Sigma L)$ thanks to \cite[Lemma 3.3]{Palu08}.  Thus $\Ker \beta_{L,M}(V,?)$ and $(T)(M,\Sigma L)$ have the same (finite) dimension, and so they are equal.
}

\section*{Acknowledgements}
The first and second authors were supported by the French ANR grant CHARMS (ANR-19-CE40-0017-02).  The second author was supported by the Institut Universitaire de France (IUF).  The third author was supported by the  National Natural Science Foundation of China (Grant No. 12271347).  The final stages of this project were completed while the first and second author were participating in a trimester programme at the Isaac Newton Institute.  The authors would like to thank the Isaac Newton Institute for Mathematical Sciences, Cambridge, for support and hospitality during the programme \emph{Cluster algebras and representation theory} where work on this paper was undertaken. This work was supported by EPSRC grant no EP/R014604/1.  

We would like to thank Karin Baur, Eleonore Faber, Ana Garcia Elsener, Alastair King, Matthew Pressland and Khrystyna Serhiyenko for discussions about applications to cluster algebras and friezes.

\bibliographystyle{alpha} 
\bibliography{cluster.bib}

\end{document}